\documentclass[11pt, reqno]{article}

\usepackage{amsmath}
\usepackage{tikz}
\usepackage{mathdots}
\usepackage{yhmath}
\usepackage{cancel}
\usepackage{color}
\usepackage{siunitx}
\usepackage{array}
\usepackage{multirow}
\usepackage{amssymb}
\usepackage{gensymb}
\usepackage{tabularx}
\usepackage{booktabs}
\usetikzlibrary{fadings}
\usetikzlibrary{patterns}
\usetikzlibrary{shadows.blur}
\usetikzlibrary{shapes}

\def\argmin{\mathop{\rm argmin}}

\RequirePackage{latexsym} \RequirePackage{amsthm}
\usepackage{enumerate}
\usepackage{dsfont}
\usepackage{mathrsfs}
\numberwithin{equation}{section}
\newtheorem{conjecture} {\sc  Conjecture\rm} [section]

\newtheorem{preremark}[conjecture]{Remark}
\newenvironment{remark}%
  {\begin{preremark}}{\end{preremark}}

\newtheorem{predefinition}[conjecture]{Definition}
\newenvironment{definition}%
  {\begin{predefinition}}{\end{predefinition}}

\newtheorem{prelemma}[conjecture]{Lemma}
\newenvironment{lemma}%
 {\begin{prelemma}}{\end{prelemma}}
\newtheorem{preproposition}[conjecture]{Proposition}
\newenvironment{proposition}%
 {\begin{preproposition}}{\end{preproposition}}

\newtheorem{precorollary}[conjecture]{Corollary}
\newenvironment{corollary}%
  {\begin{precorollary}}{\end{precorollary}}

\newtheorem{pretheorem}[conjecture]{Theorem}
\newenvironment{theorem}%
 {\begin{pretheorem}}{\end{pretheorem}}

\begin{document}

\title{On the $ \Gamma -$convergence of the Allen-Cahn functional with boundary conditions}

\author{{Dimitrios Gazoulis}\thanks{Department of Mathematics and Applied Mathematics, University of Crete, 70013 Heraklion,
Greece}\,\,\thanks{Institute of Applied and Computational Mathematics, Foundation for Research and Technology-Hellas\hspace{10ex}(dgazoulis@math.uoa.gr)}}

\date{}

\maketitle

%\author{Second author name \surname{Surname}}

%\address{Second author address \email{xxxx@xxxx.xxx.xx}}

\begin{abstract}
We study minimizers of the Allen-Cahn system. We consider the $ \varepsilon -$energy functional with Dirichlet values and we establish the $ \Gamma $-limit. The minimizers of the limiting functional are closely related to minimizing partitions of the domain. Finally, utilizing that the triod and the straight line are the only minimal cones in the plane together with regularity results for minimal curves, we determine the precise structure of the minimizers of the limiting functional, and thus the limit of minimizers of the $ \varepsilon $-energy functional as $ \varepsilon \rightarrow 0 $.
\end{abstract}

%\keywords{$ \Gamma $-convergence, Allen-Cahn system, Dirichlet Boundary conditions, Limiting minimizers, Minimizing partitions}

%\classification{35J50, 35J57}

\section{Introduction}

In this work we are concerned with the study of vector minimizers of the Allen-Cahn $ \varepsilon $-functional,
\begin{equation}\label{EnergyFunctional}
\begin{gathered}
J_{\varepsilon} (u, \Omega) := \int_{\Omega} \left( \frac{\varepsilon}{2} |\nabla u|^2 + \frac{1}{\varepsilon} W(u) \right) dx , \\
u : \Omega \rightarrow \mathbb{R}^m ,
\end{gathered}
\end{equation}
where $ \Omega \subset {\mathbb{R}}^n $ is an open set and $ W $ is a $ N $-well potential with $ N $ global minima.

Let
\begin{equation}\label{MinProblem}
u_{\varepsilon} := \argmin\limits_{v \in W^{1,2}(\Omega ; {\mathbb{R}}^m)} \lbrace J_{\varepsilon}(v, \Omega) : v|_{\partial \Omega} = g_{\varepsilon}|_{\partial \Omega} \rbrace \;,\; \textrm{where} \;\; g_{\varepsilon} \in W^{1,2} (\Omega ; {\mathbb{R}}^m).
\end{equation}
Thus $ u_{\varepsilon} \in W^{1,2}(\Omega ; {\mathbb{R}}^m) $ is a weak solution of the system
\begin{equation}\label{EulerLagrange}
\begin{cases}
\varepsilon \Delta u_{\varepsilon} - \frac{1}{\varepsilon} W_u (u_{\varepsilon}) =0 \;\:,\; \textrm{in} \;\: \Omega, \\ u_\varepsilon = g_\varepsilon \;\:,\; \textrm{on} \;\: \partial \Omega,
\end{cases}
\end{equation}
%$ \\ $

We study the asymptotic behavior of $ u_\varepsilon $ within the framework of $ \Gamma $-convergence. Moreover, we analyze the relationship between minimizers of the Allen-Cahn system and minimizing partitions subject to Dirichlet boundary conditions. For some particular assumptions on the limiting boundary conditions, we will prove uniqueness for the limiting geometric problem and we will determine the structure of the minimizers of the limiting functional.

\subsection{Main Results}

\begin{center}
\textbf{\underline{Hypothesis on $ W $:}}
\end{center}
%$ \\ $
\textbf{(H1)} $ W \in C^{1,\alpha}_{loc}({\mathbb{R}}^m ; [0,+ \infty)) \;,\; \lbrace W=0 \rbrace = \lbrace a_1,a_2,...,a_N \rbrace \;,\: N \in \mathbb{N} \;\:, a_i $ are the global minima of $ W $. Assume also that
\begin{align*}
W_u(u) \cdot u > 0 \;\: \textrm{and} \;\: W(u)\geq c_1 |u|^2 \;,\:  \textrm{if}\;\: |u |>M .
\end{align*}

\begin{center}
\textbf{\underline{Hypothesis on the Dirichlet Data:}}
\end{center}
%$ \\ $
\textbf{(H2)}\textbf{(i)} $ |g_{\varepsilon}| \leq M \;,\: g_\varepsilon \stackrel{L^1(\Omega)}{\longrightarrow} g_0 $ and $ J_\varepsilon(g_\varepsilon, \Omega_{\rho_0} \setminus \Omega) \leq C \: , $ where $ \partial \Omega $ is Lipschitz and $ \Omega_{\rho_0} $ is a small dilation of $ \Omega \;,\: \rho_0 >1 $, in which $ g_\varepsilon $ is extended $(C ,\: M \; \textrm{indep. of}\;\: \varepsilon). \\ $
And either

$ \; $\textbf{(ii)} $ g_\varepsilon \in C^{1,\alpha}(\overline{\Omega}) \;,\: | g_{\varepsilon} |_{1,\alpha} \leq \dfrac{M}{\varepsilon} $ and $ \partial \Omega $ is $ C^2 $, where we denote with $ | \cdot |_{1, \alpha} $ as the $ C^{1, \alpha} $ norm. $ \\ $
%$ \;\;\: $
Or \textbf{(ii')} $ g_\varepsilon \in H^1( \Omega) $ and $J_\varepsilon (u_\varepsilon, \Omega) \leq C $.
$ \\ $

For $ i \neq j \;\:,\; i,j \in \lbrace 1,2,...,N \rbrace $, let $ U \in W^{1,2}( \mathbb{R};\mathbb{R}^m) $ be the 1D minimizer of the action
\begin{equation}\label{EnergyofConnection}
\begin{gathered}
\sigma_{ij}:= \min \int_{-\infty}^{+\infty} \left( \frac{1}{2} |U'|^2 + W(U) \right) dt < +\infty \;\;\;, \\ \lim_{t \rightarrow - \infty} U(t) = a_i \;\;,\;\; \lim_{t \rightarrow + \infty} U(t) = a_j \;\:,\; U( \mathbb{R}) \in \mathbb{R}^m \setminus \lbrace W=0 \rbrace
\end{gathered}
\end{equation}
where $ U $ is a connection that connects $ a_i $ to $ a_j \;,\: i,j \in \lbrace 1,2,...,N \rbrace $. 

The existence of such geodesics has been proved under minimal assumptions on the potential $ W $ in \cite{ZunigaSternberg}.
%$ \sigma_{ij} $ satisfy $ \sigma_{ij} = \sigma_{ji} \;,\; \sigma_{ii}=0 $ and the strict triangle inequality $ \sigma_{ij} < \sigma_{ik} + \sigma_{kj} $ for $ k \neq i,j. $
$ \\ $

Let $ J_\varepsilon $ defined in \eqref{EnergyFunctional}, we define
\begin{equation}\label{EneryFunctionalwithBC}
\tilde{J}_{\varepsilon}(u, \Omega) := \begin{cases} J_{\varepsilon}(u,\Omega) \;\:,\; \textrm{if} \;\: u = g_{\varepsilon} \;\: \textrm{on} \;\: \Omega_{\rho_0} \setminus \Omega \;\:,\; u \in H^1_{loc}(\mathbb{R}^n;\mathbb{R}^m) \\ + \infty \;\;\;\;\;\;\;\;\;,\;\; \textrm{otherwise}
\end{cases}
\end{equation}
where $ \Omega \subset \Omega_{\rho_0} $ as in \textbf{(H2)(i)} and let
\begin{equation}\label{LimitingEneryFunctional}
J_0(u, \Omega) := \sum_{1 \leq i < j \leq N} \sigma_{ij} \mathcal{H}^{n-1} (\partial^* \Omega_i \cap \partial^* \Omega_j \cap \Omega) =  \sum_{1 \leq i < j \leq N} \sigma_{ij} \mathcal{H}^{n-1} (S_{ij}(u) \cap \Omega),
\end{equation}
where $ S_{ij}(u):= \partial^* \lbrace u=a_i \rbrace \cap \partial^* \lbrace u=a_j \rbrace \;\:,\; u \in BV(\Omega;\lbrace a_1,a_2,...,a_N \rbrace) $ and we denote as $ \partial^* \Omega_k $ the reduced boundary of $ \Omega_k $.

Finally we define the limiting functional subject to the limiting boundary conditions
\begin{equation}\label{tildeJ_0}
\tilde{J}_0 (u, \Omega) := \begin{cases} J_0(u,\Omega) \;\:,\; \textrm{if} \;\: u \in BV(\Omega;\lbrace a_1,a_2,...,a_N \rbrace) \;\: \textrm{and} \; u= g_0 \;\: \textrm{on} \;\: \Omega_{\rho_0} \setminus \Omega \\ + \infty \;\;\;\;\;\;\;\;\;,\;\; \textrm{otherwise}
\end{cases}
\end{equation}

We can write $ J_\varepsilon,J_0, \tilde{J}_\varepsilon , \tilde{J}_0 : L^1(\Omega; \mathbb{R}^n) \rightarrow \overline{\mathbb{R}} $, where $ \overline{\mathbb{R}} = \mathbb{R} \cup \lbrace \infty \rbrace $ and the $ \Gamma $-convergence will be with respect to the $ L^1 $ topology.

Our first main result is the following $ \\ $

\begin{theorem}\label{Theorem1} Let $ J_{\varepsilon} $ be defined by \eqref{EnergyFunctional} and $ \tilde{J}_{\varepsilon} \;,\: \tilde{J}_0 $ defined in \eqref{EneryFunctionalwithBC} and \eqref{tildeJ_0} respectively.

Then
\begin{equation}\label{GammaLimitwithBC}
\Gamma- \lim_{\varepsilon \rightarrow 0} \tilde{J}_{\varepsilon}(u,\Omega) = \tilde{J}_0(u,\overline{\Omega}) .
\end{equation}
%where we extend $ u $ by setting $ u = g_0 $ on $ \partial \Omega . \\ $
\end{theorem}
$ \\ $

\begin{remark}\label{Remark1}
Note that the domain of $ \tilde{J}_0 $ is the closure of $ \Omega $, which means that there is a boundary term (see also (2.9) in \cite{ORS} for the analog in the scalar case). More precisely, by Proposition \ref{Proposition1} and Theorem 5.8 in \cite{EG} we can write
\begin{equation}
\begin{gathered}
\tilde{J}_0(u, \overline{\Omega}) = \frac{1}{2} \sum_{i=1}^N \int_{\overline{\Omega}} | D ( \phi_i \circ u) | \\ = \frac{1}{2} \sum_{i=1}^N \int_{\Omega} | D ( \phi_i \circ u) | + \frac{1}{2} \sum_{i=1}^N \int_{\partial \Omega} | T(\phi_i \circ u) - T(\phi_i \circ g_0) | \: d \mathcal{H}^{n-1} \\
\textrm{where} \;\: \phi_i \;\: \textrm{defined in} \;\: \eqref{Riemannianmetric} \;\: \textrm{and} \;\: T \;\: \textrm{is the trace operator for} \;\: BV \;\: \textrm{functions}.
\end{gathered}
\end{equation}
\end{remark}
$ \\ $

The overview of the strategy of the proof of Theorem \ref{Theorem1} is as follows. First we observe that the $ \Gamma- $limit established in \cite{Baldo}, in particular Theorem 2.5, holds also without the mass constraint (see Theorem \ref{GammaLimitThmBaldo} in Preliminaries section). Next, we apply a similar strategy to that of \cite[Theorem 3.7]{ABP} in which there is a $ \Gamma $-convergence result with boundary conditions in the scalar case which states that we can incorporate the constraint of Dirichlet values in the $ \Gamma- $limit, provided that this $ \Gamma- $limit is determined.
Since by Theorem \ref{GammaLimitThmBaldo} we have that $ J_\varepsilon \;\: \Gamma $-converges to $ J_0 $, we establish the $ \Gamma $-limit of $ \tilde{J}_\varepsilon $, that is, the $ \Gamma- $limit of the functional $ J_\varepsilon $ with the constraint of Dirichlet values. For the proof of the $ \Gamma- $limit we can assume either \textbf{(H2)(ii)} or \textbf{(H2)(ii')}.
$ \\ \\ $

Next, we study the solution of the geometric minimization problem that arise from the limiting functional.

In order to obtain precise information about the minimizer of the limiting functional $ \tilde{J}_0 (u, \overline{B}_1) \:, \; B_1 \subset \mathbb{R}^2 $, we impose that the limiting boundary conditions $ g_0 $ have connected phases. So we assume, $ \\ \\ $
\textbf{(H2) (iii)} Let $ g_0 = \sum_{i=1}^3 a_i \: \chi_{I_i}(\theta) \;,\; \theta \in [0, 2 \pi) \;,\; I_i \subset [0, 2 \pi) \;\:,\; \cup_{i=1}^3 I_i = [0, 2 \pi) $ be the limit of $ g_\varepsilon $. Assume that $ I_i $ are connected and that
\begin{align*}
\theta_0  < \frac{2 \pi}{3} \;,\: \textrm{where} \;\: \theta_0 \;\: \textrm{is the largest angle of the points} \;\: p_i = \partial I_k \cap \partial I_l \\ k \neq l \;,\: i \in \lbrace 1,2,3 \rbrace \setminus \lbrace k,l \rbrace . \;\;\;\;\;\;\;\;\;\;\;\;\;\;\;\;\;\;\;\;\;\;\;\;\;\;\;\;\;\;\;\;\;\;
\end{align*}
%$ \\ $

The assumption $ \theta_0 < \frac{2 \pi}{3} $ arises from the Proposition 3.2 in \cite{Morgan} that we utilize for the proof (see Proposition \ref{PropMorgan} in Preliminaries section) and guarantees that the boundary of the partition defined by the minimizer will be line segments meeting at a point inside $ B_1 $. 

Our second main result is the following 
$ \\ $

\begin{theorem}\label{TheoremTriod} Let $ u_0 = a_1 \chi_{\Omega_1} +a_2 \chi_{\Omega_2} + a_3 \chi_{\Omega_3} $ be a minimizer of $ \tilde{J}_0(u,\overline{B}_1) $ subject to the limiting Dirichlet values \textbf{(H2)(iii)}. 

Then the minimizer is unique and in addition,
\begin{equation}\label{ThmTriodEq}
\partial \Omega_i \cap \partial \Omega_j \;\: \textrm{are line segments meeting at} \;\: 120^o \;\: \textrm{in a point in} \;\: B_1 \;\: (i \neq j).
\end{equation}
$ \\ $
\end{theorem}

For proving Theorem \ref{TheoremTriod}, we first prove that the partition defined by $ u_0 $ is $ (M,0, \delta)- $minimal as in the Definition 2.1 in \cite{Morgan} (see Definition \ref{M,0,deltaminimalSets}). This is proved by a comparison argument by defining a Lipschitz perturbation of the partition of the minimizer with strictly less energy. Then, by utilizing a uniqueness result for $ (M, 0 ,\delta)- $minimal sets in \cite{Morgan} (see Proposition \ref{PropMorgan}), we can conclude that the minimizer of the limiting energy is unique and the boundaries of the partition that the minimizer defines are are line segments meeting at $ 120^o $ degrees in an interior point of the unit disc.

In the last subsection, we note that the result in Theorem \ref{TheoremTriod} can be extended also to the mass constraint case (see \cite{Baldo}). However, in this case the uniqueness will be up to rigid motions of the disc (see Theorem 3.6 and Theorem 4.1 in \cite{CR}).
$ \\ $
\subsection{Previous fundamental contributions}

We will now briefly introduce some of the well known results in the scalar case. The notion of $ \Gamma $-convergence was introduced by E. De Giorgi and T. Franzoni in \cite{DeGFranzoni} and in particular relates phase transition type problems with the theory of minimal surfaces. One additional application of $ \Gamma $-convergence is the proof of existence of minimizers of a limiting functional, say $ F_0 $, by utilizing an appropriate sequence of functionals $ F_\varepsilon $ that we know they admit a minimizer and the $ \Gamma $-limit of $ F_\varepsilon $ is $ F_0 $. And also vice versa (\cite{KS}), we can obtain information for the $ F_\varepsilon $ energy functional from the properties of minimizers of the limiting functional $ F_0 $. We can think of this notion as a generalization of the Direct Method in the Calculus of Variations i.e. if $ F_0 $ is lower semicontinuous and coercive we can take $ F_\varepsilon = F_0 $ and then $ \Gamma- $lim $ F_\varepsilon =F_0 $.

There are many other ways of thinking of this notion, such as a proper tool in finding the limiting functional among a sequence of functionals.

Let $ X $ be the space of the measurable functions $ u : \Omega \subset \mathbb{R}^n \rightarrow \mathbb{R} $ endowed with the $ L^1 $ norm and
\begin{align*}
F_\varepsilon(u,\Omega) := \begin{cases} \int_\Omega \frac{\varepsilon}{2} | \nabla u|^2 + \frac{1}{\varepsilon}W(u) dx \;\;,\; u \in W^{1,2} (\Omega ; \mathbb{R}) \cap X
\\ + \infty \;\;\;\;\;\;\;\;\;\;\;\;\;\;\;\;\;\;\;\;\;\;\;\;\;\;\;\;\;, \; \textrm{elsewhere in} \;\: X \end{cases}
\\
F_0(u,\Omega) := \begin{cases} \sigma \mathcal{H}^{n-1} (Su) \;\;,\; u \in SBV(\Omega ; \lbrace -1,1 \rbrace) \cap X \\ + \infty  \;\;\;\;\;\;\;\;\;\;\;\;\;, \; \textrm{elsewhere in} \;\: X 
\end{cases} \;\;\;\;\;\;\;\;\;\;
\\ \textrm{where} \;\: W : \mathbb{R} \rightarrow [0, + \infty ) \;,\; \lbrace W =0 \rbrace = \lbrace -1 , 1 \rbrace \;,\; \sigma = \int_{-1}^1 \sqrt{2 W(u)}du \\
{and} \;\: Su \;\: \textrm{is the singular set of the SBV function} \;\: u. \;\;\;\;\;\;\;\;\;\;\;\;\;\;\;\;
\end{align*}

Let now $ u_\varepsilon $ be a minimizer of $ F_\varepsilon $ subject to a mass constraint, that is, $ \int_{\Omega}u = V \in (0, |\Omega|) $. The asymptotic behavior of $ u_\varepsilon $ was first studied by Modica and Mortola in \cite{MM} and by Modica in \cite{Modica,Modica2}. Also, later Sternberg \cite{S} generalized these results for minimizers with volume constraint. Furthermore, Owen, Rubinstein and Sternberg in \cite{ORS} and Ansini, Braides and Piat in \cite{ABP}, among others, studied the asymptotic behavior of the minimizers subject to Dirichlet values for the scalar case.

As mentioned previously, one of the most important outcomes of $ \Gamma $-convergence in the scalar phase transition type problems is the relationship with minimal surfaces. More precisely, the well known theorem of Modica and Mortola states that the $ \varepsilon $-energy functional of the Allen-Cahn equation $ \Gamma $-converges to the perimeter functional that measures the perimeter of the interface between the phases (i.e. $ \Gamma - \textrm{lim} \: F_\varepsilon = F_0 $). So the interfaces of the limiting problem will be minimal surfaces.

This relationship is deeper as indicated in the De Giorgi conjecture (see \cite{DeGiorgi}) which states that the level sets of global entire solutions of the scalar Allen-Cahn equation that are bounded and strictly monotone with respect to $ x_n $, are hyperplanes if $ n \leq 8 $. The relationship with the Bernstein problem for minimal graphs is the reason why $ n \leq 8 $ appears in the conjecture. The $ \Gamma $-limit of the $ \varepsilon $-energy functional of the Allen-Cahn equation is a possible motivation behind the conjecture.

In addition, Baldo in \cite{Baldo} and Fonseca and Tartar in \cite{FT} extended the $ \Gamma $-convergence analysis for the phase transition type problems to the vector case subject to a mass constraint and the limiting functional measures the perimeter of the interfaces separating the phases, and thus there is a relationship with the problem of minimizing partitions. In section 5 we analyze this in the set up of Dirichlet boundary conditions. Furthermore, the general vector-valued coupled case has been thoroughly studied in the works of Borroso-Fonseca and Fonseca-Popovici in \cite{BF} and \cite{FP} respectively.

There are many other fundamental contributions on the subject, such as the works of Gurtin \cite{Gurtin,Gurtin2}, Gurtin and Matano \cite{GurtinMatano} on the Modica-Mortola functional and its connection with materials science, the work of Hutchingson and Tonegawa on the convergence of critical points in \cite{HT}, the work of Bouchitté \cite{Bouchitte} and of Cristoferi and Gravina \cite{CG} on space-dependent wells and extensions on general metric spaces in the work of Ambrosio in \cite{Ambrosio}. Several extensions to the non-local case and fractional setting have also been studied by Alberti-Bellettini in \cite{AB}, by Alberti-Bouchitté-Seppecher in \cite{ABS} and by Savin-Valdinoci in \cite{SV} among others.

$ \\ $

\textbf{Acknowledgements:} I wish to thank my advisor Professor Nicholas Alikakos for his guidance and for suggesting this topic as a part of my thesis for the Department of Mathematics and Applied Mathematics at the University of Crete. Also, I would like to thank Professor P. Sternberg and Professor F. Morgan for their valuable comments on a previous version of this paper, which let to various improvements. Finally, I would like to thank the anonymous referee for their valuable suggestions, which not only enhanced the presentation but also significantly improved the quality of the paper by relaxing some of the assumptions in our results.

$ \\ $
\section{Preliminaries}
\subsection{Specialized definitions and theorems for the $ \Gamma- $limit}

First, we will define the supremum of measures that allow us to express the limiting functional in an alternative way. Let $ \mu $ and $ \nu $ be two regular Borel measures on $ \Omega $ we denote by $ \mu \bigvee \nu $ the smallest regular positive measure which is greater than or equal to $ \mu $ and $ \nu $ on all borel subsets of $ \Omega $, for $ \mu \;,\: \nu $ being two regular positive Borel measures on $ \Omega . $ We have
\begin{align*}
(\mu \bigvee \nu)(\Omega) := \sup \lbrace \mu (A) + \nu(B) : A \cap B = \emptyset, \; A\cup B \subset \Omega, \; A \;\: \textrm{and} \;\: B \;\: \textrm{are open sets in} \;\: \Omega \rbrace .
\end{align*}

Now let
\begin{align*}
\bigvee_{k=1}^N \int_{\Omega} | D (\phi_k \circ u_0 ) | := \sup \lbrace \sum_{k=1}^N \int_{A_k} | D(\phi_k \circ u_0 ) | : \cup_{k=1}^N A_k \subset \Omega ,\\ A_i \cap A_j = \emptyset \:,\: i \neq j, \; A_i \;\: \textrm{open sets in} \;\: \Omega \rbrace.
\end{align*}

$ \\ $

We will now provide a Lemma from \cite{ABP} that is crucial in the description of the behavior of the $ \Gamma- $limit with respect to the set variable.
Let $ \Omega \subset \mathbb{R}^n $ be an open set. We denote by $ \mathcal{A}_{\Omega} $ the family of all bounded open subsets of $ \Omega $.

\begin{lemma}\label{LemmainABPcutoff} (\cite{ABP})
Let $ J_\varepsilon $ defined in \eqref{EnergyFunctional}. Then for every $ \varepsilon >0 $, for every bounded open sets $ U \:,\; U' \:,\; V $, with $ U \subset \subset U' $, and for every $ u,v \in L^1_{loc}(\mathbb{R}^n) $, there exist a cut-off function $ \phi $ related to $ U $ and $ U' $, which may depend on $ \varepsilon \;,\: U \;,\: U' \;,\: V \;,\: u \;,\: v $ such that
\begin{align*}
J_\varepsilon ( \phi u + (1- \phi) v, U \cup V) \leq J_\varepsilon (u, U') + J_\varepsilon (v,V) + \delta_\varepsilon (u,v,U,U',V),
\end{align*}
where $ \delta_\varepsilon : L^1_{loc}(\mathbb{R}^n)^2 \times \mathcal{A}^3_{\Omega} \rightarrow [0, + \infty) $ are functions depending only on $ \varepsilon $ and $ J_\varepsilon $ such that
\begin{align*}
\lim_{\varepsilon \rightarrow 0} \delta_{\varepsilon} (u_\varepsilon, v_\varepsilon, U,U',V) =0,
\end{align*}
whenever $  U \;,\: U' \;,\: V \in \mathcal{A}_{\Omega} \:,\; U \subset \subset U' $ and $ u_\varepsilon \:, \: v_\varepsilon \in L^1_{loc}(\mathbb{R}^n) $ have the same limit as $ \varepsilon \rightarrow 0 $ in $ L^1( (U' \setminus \overline{U}) \cap V ) $ and satisfy
\begin{align*}
\sup_{\varepsilon>0}( J_\varepsilon (u_\varepsilon,U') + J_\varepsilon(v_\varepsilon,V)) < + \infty .
\end{align*}
\end{lemma}
The above result is Lemma 3.2 in \cite{ABP} and has been proved in the scalar case. The proof also works in the vector case with minor modifications. In \cite{ABP}, there is an assumption on $ W $, namely $ W \leq c( | u |^\gamma +1) $ with $ \gamma \geq 2 $ (see (2.2) in \cite{ABP}). This assumption however is only utilized in the proof of Lemma 2.1 above to apply the dominated convergence theorem in the last equation. In our case this assumption is not necessary since $ W(u_\varepsilon) $ and $ W (g_\varepsilon) $ are uniformly bounded (see \textbf{(H2)(i)} and Lemma \ref{Lemma1}). In fact, the only reason we assume in \textbf{(H1)} that $ W(u) \geq c_1 | u |^2 $ for $ | u |>M $ is to apply the above Lemma.
$ \\ $

In \cite{Baldo} it has been proved that $ J_{\varepsilon} \;\: \Gamma -$converges to $ J_0 $ with mass constraint, but it also holds without mass constraint (see Theorem 2.5). We will point out this more clearly in the proof of Theorem \ref{Theorem1}. In particular, it holds
$ \\ $

\begin{theorem}\label{GammaLimitThmBaldo} (\cite{Baldo})
Let $ J_\varepsilon $ defined in \eqref{EnergyFunctional} and $ J_0 $ defined in \eqref{LimitingEneryFunctional}. Then $ \Gamma- \lim_{\varepsilon \rightarrow 0} J_{\varepsilon}(u,\Omega) = J_0(u, \Omega) $ in $ L^1(\Omega; \mathbb{R}^m). $ That is, for every $ u \in L^1(\Omega; \mathbb{R}^m) $, we have the following two conditions: 
$ \\ $ (i) If $ \lbrace v_\varepsilon \rbrace \subset L^1(\Omega; \mathbb{R}^m) $ is any sequence converging to $ u $ in $ L^1 $, then
\begin{equation}\label{GammaLiminfBaldoEq}
\liminf_{\varepsilon \rightarrow 0} J_\varepsilon (v_\varepsilon, \Omega) \geq J_0 (u, \Omega),
\end{equation}
and $ \\ $
(ii) There exist a sequence $ \lbrace w_\varepsilon \rbrace \subset L^1(\Omega; \mathbb{R}^m) $ converging to $ u $ in $ L^1 $ such that
\begin{equation}\label{GammaLimsupBaldoEq}
\lim_{\varepsilon \rightarrow 0} J_\varepsilon (w_\varepsilon, \Omega) = J_0 (u, \Omega).
\end{equation}
\end{theorem}
$ \\ $

\begin{remark}\label{RmkTechnicalAssumptionWBaldo} We note that in \cite{Baldo}, there is also a technical assumption for the potential $ W $ (see (1.2) in p.70). However for the proof of the $ \Gamma- $limit this assumption is only utilized for the proof of the liminf inequality in order to obtain the equiboundedness of the minimizers $ u_\varepsilon $ (see proof of (2.8) in \cite{Baldo}). However in our case we obtain equiboundedness from Lemma \ref{Lemma1} in the following section. Therefore in our case this assumption is dismissed.
\end{remark}

$ \\ $
\subsection{Specialized definitions and theorems for the Geometric problem}

In addition, we introduce the notion of $ (M, 0 ,\delta) $-minimality as defined in \cite{Morgan} together with a Proposition that certifies the shortest network connecting three given points in $ \mathbb{R}^2 $ as uniquely minimizing in the context of $ (M, 0 ,\delta) -$ minimal sets. This characterization is one of the ingredients for the solution of the geometric minimization problem in the last section. In fact, in \cite{Morgan} the more general notion of $ (M, \varepsilon ,\delta) $-minimality (or $ (M, c r^\alpha ,\delta) $-minimality) is introduced and regularity results for such sets are established.
Particularly, $ (M, 0 ,\delta) -$ minimality implies $ (M, c r^\alpha ,\delta) $-minimality (see \cite{Morgan}).

\begin{definition}\label{M,0,deltaminimalSets} (\cite{Morgan}) Let $ K \subset \mathbb{R}^n $ be a closed set and fix $ \delta >0 $. Consider $ S \subset \mathbb{R}^n \setminus K $ be a nonempty bounded set of finite $ m $-dimensional Hausdorff measure. $ S $ is $ (M, 0 ,\delta) $-minimal if $ S = spt( \mathcal{H}^m \lfloor S) \setminus K $ and
\begin{align*}
\mathcal{H}^m (S \cap W) \leq \mathcal{H}^m ( \phi(S \cap W)),
\end{align*}
whenever
\begin{align*}
(a) \;\: \phi: \mathbb{R}^n \rightarrow \mathbb{R}^n \;\: \textrm{is lipschitzian}, \;\;\;\;\;\;\;\;\;\;\;\;\;\;\;\;\;\;\;\;\;\;\;\;\;\;\;\;\;\; \\
(b) \;\: W = \mathbb{R}^n \cap \lbrace z \; : \; \phi(z) \neq z \rbrace , \;\;\;\;\;\;\;\;\;\;\;\;\;\;\;\;\;\;\;\;\;\;\;\;\;\;\;\;\;\;\;\; \\
(c) \;\: \textrm{diam}(W \cup \phi(W)) < \delta , \;\;\;\;\;\;\;\;\;\;\;\;\;\:\:\;\;\;\;\;\;\;\;\;\;\;\;\;\;\;\;\;\;\;\;\;\;\;\;\; \\
(d) \;\: \textrm{dist}(W \cup \phi(W), K) > 0. \;\;\;\;\;\;\;\;\;\;\;\;\;\;\;\;\;\;\;\;\;\;\;\;\;\;\;\;\;\;\;\;\;\;\;\;
\end{align*}
\end{definition}

$ \\ $

\begin{proposition}\label{PropMorgan} (\cite{Morgan}) Let $ K = \lbrace p_1,p_2 ,p_3 \rbrace $ be the vertices of a triangle in the open $ \delta $-ball $ B(0,\delta) \subset \mathbb{R}^2 $, with largest angle $ \theta $ for some fixed $ \delta >0 $. Then there exist a unique smallest $ (M,0 , \delta)- $minimal set in $ B(0,\delta) $ with closure containing $ K $, in particular:
\begin{align*}
(a) \;\: \textrm{if} \;\: \theta \geq 120^o \: ,\;\: \textrm{the two shortest sides of the triangle} ; \;\;\;\;\;\;\;\;\;\;\;\;\;\;\;\;\;\;\;\;\;\;\;\; \\
(b) \;\: \textrm{if} \;\: \theta < 120^o \: ,\;\: \textrm{segments from three vertices meeting at} \;\: 120^o . \;\;\;\;\;\;\;\;\;\;\;
\end{align*}
Here by the ``unique smallest'' we mean any other such $ (M,0 , \delta)- $minimal set $ S $ has larger one-dimensional Hausdorff measure.
\end{proposition}
$ \\ $

We now state a well known Bernstein-type theorem in $ \mathbb{R}^2 . \\ $

\begin{theorem}\label{BernsteinThmforPartitionsinthePlane} (\cite{A}) Let $ A $ be a complete minimizing partition in $ \mathbb{R}^2 $ with $ N=3 $ (three phases), with surface tension coefficients satisfying
\begin{equation}\label{triangleinequality}
\sigma_{ik} < \sigma_{ij} + \sigma_{jk} \;\;\;,\; \textrm{for} \;\: j \neq i,k \;\: \textrm{with} \;\: i,j,k \in \lbrace 1,2,3 \rbrace.
\end{equation}
Then $ \partial A $ is a triod.
\end{theorem}

For a proof and related material we refer to \cite{WhiteNotes} and the expository \cite{A}. $ \\ $

\section{Basic Lemmas}
%$ \\ $

\begin{lemma}\label{Lemma1} For every critical point $ u_{\varepsilon} \in W^{1,2}(\Omega;\mathbb{R}^m) $, satisfying \eqref{EulerLagrange} weakly together with the assumptions \textbf{(H1)} and \textbf{(H2)(i),(ii)}, it holds
\begin{align*}
|| u_{\varepsilon}||_{L^{\infty}} <M \;\;\;\; \textrm{and} \;\;\;\; || \nabla u_{\varepsilon}||_{L^{\infty}} < \frac{\tilde{C}}{\varepsilon}.
\end{align*}
\end{lemma}
\begin{proof} $ \;\: $
By linear elliptic theory, we have that $ u_{\varepsilon} \in C^2(\Omega;{\mathbb{R}}^m) $ (see for example Theorem 6.13 in \cite{GT}). Set $ v_{\varepsilon}(x) = |u_{\varepsilon}(x)|^2 $, then
\begin{align*}
\Delta v_{\varepsilon} = 2W_u(u_{\varepsilon}) \cdot u_{\varepsilon} + 2 | \nabla u_{\varepsilon}|^2 >0 \;\;\;\; for \;\; | u_{\varepsilon}| >M \:,
\end{align*}
Hence $ \max_{\Omega} | u_{\varepsilon}|^2 \leq M^2 $.

On the other hand (from \textbf{(H2)}), $ \max_{\partial \Omega} | u_{\varepsilon} | \leq M $. Thus $ \max_{\overline{\Omega}} | u_{\varepsilon}| \leq M . \\ $ 
For the gradient bound, consider the rescaled problem $ y= \frac{x}{\varepsilon} $, denote by $ \tilde{u} \:,\: \tilde{g} $ the rescaled $ u_{\varepsilon} \:,\: g_{\varepsilon} $, so by elliptic regularity (see for example Theorem 8.33 in \cite{GT}),
%By (8.86) in \cite{GT},

\begin{align*}
| \tilde{u} |_{1,\alpha} \leq C (||\tilde{u}||_{L^{\infty}} + | \tilde{g}|_{1,\alpha}) \leq 2CM \\
\Rightarrow || \nabla \tilde{u} ||_{L^{\infty}} \leq 2CM \Rightarrow | \nabla u_{\varepsilon} | \leq \frac{\tilde{C}}{\varepsilon} .
\end{align*}
\end{proof}

\begin{lemma}\label{Lemma2} Let $ u_\varepsilon $ defined in \eqref{MinProblem}, then 
\begin{align*}
J_{\varepsilon} (u_{\varepsilon}) = \int_{\Omega} \left( \frac{\varepsilon}{2} |\nabla u_{\varepsilon}|^2 + \frac{1}{\varepsilon} W(u_{\varepsilon}) \right) dx  \leq C \;,
\end{align*}
$ C $ independent of $ \varepsilon >0 $, if $ \Omega $ is bounded.
\end{lemma}
%$ \\ $
\begin{proof} $ \;\: $
Without loss of generality we will prove Lemma \ref{Lemma2} for $ \Omega =B_1 $ (or else we can cover $ \Omega $ with finite number of unit balls and the outside part is bounded by \textbf{(H2)}(i)).

Substituting $ y = \dfrac{x}{\varepsilon} $,
\begin{align*}
J_{\varepsilon} (u_{\varepsilon}) = \int_{B_{\frac{1}{\varepsilon}}} \left( \frac{\varepsilon}{2} |\nabla_y \tilde{u}_{\varepsilon}|^2 \: \frac{1}{\varepsilon^2} + \frac{1}{\varepsilon} W(\tilde{u}_{\varepsilon}) \right) \varepsilon^n dy ,
\end{align*}
where $ \tilde{u}_{\varepsilon} = u_{\varepsilon} (\varepsilon y) $ and for $ \varepsilon = \dfrac{1}{R} $,
\begin{align*}
\Rightarrow J_{\varepsilon} (u_{\varepsilon}) = \varepsilon^{n-1}  
\int_{B_{\frac{1}{\varepsilon}}} \left( \frac{1}{2} |\nabla_y \tilde{u}_{\varepsilon}|^2 + W(\tilde{u}_{\varepsilon}) \right) dy = \frac{1}{R^{n-1}} \int_{B_R} \left( \frac{1}{2} |\nabla_y \tilde{u}_R|^2 + W(\tilde{u}_R) \right) dy \\ = \frac{1}{R^{n-1}} \tilde{J}_R (\tilde{u}_R) . \;\;\;\;\;\;\;\;\;\;\;\;\;\;\;\;\;\;\;\;\;\;\;\;\;\;\;\;\;\;\;\;
\end{align*}
So, $ \tilde{u}_R $ is minimizer of $ \tilde{J}_R (v) = \int_{B_R} (\frac{1}{2} |\nabla v|^2 + W(v) ) dx $. 

By Lemma \ref{Lemma1} applied in $ u_\varepsilon $, it holds that $ | \tilde{u}_R | ,| \nabla \tilde{u}_R | $ are uniformly bounded independent of $ R $ and via the comparison function (see \cite{AFS} p.135), for $ R>1 $
\begin{align*}
v(x) := \begin{cases} a_1 \;, \;\;\;\;\;\;\;\;\;\;\;\;\;\;\;\;\;\;\;\;\;\;\;\;\;\;\;\;\;\;\;\;\;\;\;\;\;\;\;\;\;\;\;\;\;\;\;\;\; \textrm{for} \;\: | x | \leq R-1 \\ (R- | x|)a_1 + (|x| -R +1)\tilde{u}_R(x) \;,\; \textrm{for} \;\: | x | \in (R-1,R] \\ \tilde{u}_R(x) \;, \;\;\;\;\;\;\;\;\;\;\;\;\;\;\;\;\;\;\;\;\;\;\;\;\;\;\;\;\;\;\;\;\;\;\;\;\;\;\;\;\;\;\; \textrm{for} \;\: |x| >R
\end{cases} ,
\end{align*}
we have
\begin{align*}
\tilde{J}_R (\tilde{u}_R) \leq J(v) \leq C R^{n-1} \;\;, \; C \;\: \textrm{independent of} \;\: R.
\end{align*}
Thus
\begin{align*}
J_{\varepsilon}(u_{\varepsilon}) = \frac{1}{R^{n-1}} \tilde{J}_R (\tilde{u}_R) \leq C \;\; ( C \;\: \textrm{independent of} \;\: \varepsilon >0 ).
\end{align*}
\end{proof}

\begin{lemma}\label{Lemma3} Let $ u_\varepsilon $ defined in \eqref{MinProblem}, then $ u_{\varepsilon} \stackrel{L^1}{\longrightarrow} u_0 $, along subsequences and $ u_0 \in BV(\Omega;\mathbb{R}^m)$. Moreover, $ u_0 = \sum_{i=1}^N a_i \chi_{\Omega_i} \:,\; \mathcal{H}^{n-1}(\partial^* \Omega_i) < \infty $ and $ | \Omega \setminus \cup_{i=1}^N \Omega_i | =0 $.
\end{lemma}
%$ \\ $
\begin{proof}
$ \;\: $
By Lemma \ref{Lemma1} we have that $ u_\varepsilon $ is equibounded. Now arguing as in the proof of Proposition 4.1 in \cite{Baldo} (see also Remark \ref{RmkTechnicalAssumptionWBaldo}), we obtain that $ || u_{\varepsilon} ||_{BV(\Omega;\mathbb{R}^m)} $ is uniformly bounded, $ u_\varepsilon \rightarrow u_0 $ in $ L^1 $ along subsequences and also $ u_0 \in BV(\Omega ; \mathbb{R}^m) . $

From Lemma \ref{Lemma2}, it holds
\begin{align*}
\frac{1}{\varepsilon} \int_{\Omega} W(u_{\varepsilon}(x))dx \leq C \;\:\;( C \;\: \textrm{independent of} \;\: \varepsilon >0).
\end{align*}
Since $ | u_{\varepsilon} | \leq M $ and $ W $ is continuous in $ \overline{B}_M \subset {\mathbb{R}}^m \: \Rightarrow W( u_{\varepsilon}) \leq \tilde{M} $, therefore by the dominated convergence theorem we obtain
\begin{align*}
\int_{\Omega} W(u_0(x)) dx =0 \Rightarrow u_0 \in \lbrace W=0 \rbrace \;\: a.e. \;\: \Rightarrow u_0 = \sum_{i=1}^N a_i \chi_{\Omega_i}
\end{align*}
where $ \chi_{\Omega_i} $ have finite perimeter since $ u_0 \in BV( \Omega ;\mathbb{R}^m) $ (see \cite{EG}). $ \\ $

The proof of Lemma \ref{Lemma3} is complete.
\end{proof}

$ \\ $

Also, $ g_0 $ takes values on $ \lbrace W=0 \rbrace $.

\begin{lemma}\label{g_0takevaluesinthewells}
Let $ g_0 $ be the limiting boundary condition of $ g_\varepsilon $. 

Then
\begin{align*}
g_0 = \sum_{i=1}^N a_i \chi_{I_i} \;\:, \; \textrm{where} \;\: I_i \;\: \textrm{have finite perimeter and} \;\: | \partial \Omega \setminus \cup_{i=1}^N I_i|=0.
\end{align*}
\end{lemma}
\begin{proof} $ \;\: $
By \textbf{(H2)(i)},
\begin{align*}
J_\varepsilon ( g_\varepsilon , \Omega_{\rho_0} \setminus \Omega) \leq C \;\;\;\;\; \\
\Rightarrow \frac{1}{\varepsilon} \int_{\Omega_{\rho_0} \setminus \Omega} W(g_\varepsilon) dx \leq C
\end{align*}
So, arguing as in the proof of Lemma \ref{Lemma3}, we have that $ g_0 \in \lbrace W=0 \rbrace $ and we conclude.
\end{proof}
%$ \\ $

\begin{proposition}\label{Proposition1} It holds that
\begin{equation}\label{Prop1}
\begin{gathered}
\int_{\Omega'} | D (\phi_k \circ u_0 ) | = \sum_{i=1, i\neq k}^N \sigma_{ik} \mathcal{H}^{n-1} ( \partial^* \Omega_k \cap \partial^* \Omega_i \cap \Omega' ) \\ k =1,2,..,N \;\:, \textrm{for every open} \;\: \Omega' \subset \Omega ,
\end{gathered}
\end{equation}
where $ \phi_k (z) = d(z,a_k) \;,\: k=1,2,...,N, $ and $ a_k $ are the zeros of $ W $ and $ d $ is the Riemannian metric derived from $ W^{1/2} $, that is
\begin{align}\label{Riemannianmetric}
d(z_1,z_2) := \inf \left\lbrace \int_0^1 \sqrt{2} W^{1/2} ( \gamma (t)) | \gamma' (t)| dt : \gamma \in C^1 ([0,1];{\mathbb{R}}^2), \gamma(0) = z_1, \: \gamma(1)=z_2 \right\rbrace .
\end{align}
\end{proposition}
%$ \\ $

\begin{proof} $ \;\: $ The proof can be found in Proposition 2.2 in \cite{Baldo}.

\end{proof}

Furthermore, reasoning as in the proof of Proposition 2.2 in \cite{Baldo} we have,
\begin{equation}\label{HausdorffMeasuresEstimate}
\bigvee_{k=1}^N \int_{\Omega} | D (\phi_k \circ u_0 ) | = \sum_{1 \leq i< j \leq N} \sigma_{ij} \mathcal{H}^1 (\partial^* \Omega_i \cap \partial^* \Omega_j \cap \Omega) = J_0(u_0, \Omega).
\end{equation}
The above equation is an alternative way to express the limiting functional.
$ \\ $

\section{Proof of the $ \Gamma $-limit}

Throughout the proof of the $ \Gamma- $limit we will assume \textbf{(H1)} and \textbf{(H2)(i),(ii)}. The proof if we assume \textbf{(H2)(ii')} instead of \textbf{(H2)(ii)} is similar with minor modifications.

\begin{proof}[Proof of Theorem~\ref{Theorem1}]
$ \\ $

We begin by proving the $ \Gamma- \liminf $ inequality.

Let $ u_\varepsilon \in L^1(\Omega; \mathbb{R}^m) $ such that $ u_\varepsilon \rightarrow u $ in $ L^1(\Omega; \mathbb{R}^m) $. If $ u_\varepsilon \notin H^1_{loc} $ or $ u_\varepsilon \neq g_\varepsilon $ on $ \Omega_{\rho_0} \setminus \Omega $, where $ \Omega \subset \Omega_{\rho_0} $ as in \textbf{(H2)(i)}, then $ \tilde{J}_\varepsilon(u_\varepsilon, \Omega) = + \infty $ and the liminf inequality holds trivially. So, let $ u_\varepsilon \in H^1_{loc}(\Omega; \mathbb{R}^m) $ such that $ u_\varepsilon \rightarrow u $ in $ L^1 $ and $ u_\varepsilon = g_\varepsilon $ on $ \Omega_{\rho_0} \setminus \Omega $.

Let $ \rho >1 $ such that $ \rho < \rho_0 $ in \textbf{(H2)(i)}, we have
\begin{equation}\label{LiminfIneqEq1}
\tilde{J}_{\varepsilon}(u_{\varepsilon},\Omega) = J_{\varepsilon}(u_{\varepsilon}, \Omega_{\rho}) - J_{\varepsilon}(g_{\varepsilon}, \Omega_{\rho} \setminus \Omega),
\end{equation}
where $ \partial \Omega_\rho \in C^2 $ since it is a small dilation of $ \Omega $ and there is a unique normal vector $ \nu \perp \partial \Omega_{\rho} $, such that each $ x \in \partial \Omega $ can be written as $ x = y + \nu(y) d \;,\: d = dist(x, \partial \Omega_{\rho}) $ (see the Appendix in \cite{GT}). 

So,
\begin{equation}\label{LiminfIneqEq2}
J_{\varepsilon}(g_{\varepsilon}, \Omega_{\rho} \setminus \Omega) = \int_1^\rho \int_{\partial \Omega_r} \left( \frac{\varepsilon}{2} | \nabla g_\varepsilon|^2 + \frac{1}{\varepsilon} W(g_\varepsilon) \right) dS dr \leq C ( \rho -1),
\end{equation}
by Fubini's Theorem and \textbf{(H2)(i)}. 

Hence, by \eqref{LiminfIneqEq1}, for every $ u_{\varepsilon} $ converging to $ u $ in $ L^1 $ such that $ u_{\varepsilon}= g_{\varepsilon} $ on $ \Omega_{\rho_0} \setminus \Omega $ and $ \liminf_{\varepsilon \rightarrow 0} \tilde{J}_{\varepsilon} (u_{\varepsilon}, \Omega) < + \infty $, we have that
\begin{equation}\label{LiminfIneqEq3}
\liminf_{\varepsilon \rightarrow 0} \tilde{J}_{\varepsilon}(u_{\varepsilon}, \Omega) \geq \liminf_{\varepsilon \rightarrow 0} J_{\varepsilon}(u_{\varepsilon},\Omega_{\rho}) - O(\rho-1).
\end{equation}
Also, by the liminf inequality for $ J_{\varepsilon} $ (see Theorem \ref{GammaLimitThmBaldo} and \eqref{HausdorffMeasuresEstimate}), we can obtain
\begin{equation}\label{LiminfIneqEq4}
\liminf_{\varepsilon \rightarrow 0} J_{\varepsilon}(u_{\varepsilon}, \Omega_\rho) \geq \sum_{1 \leq i< j \leq N} \sigma_{ij} \mathcal{H}^1 (\partial^* \Omega_i \cap \partial^* \Omega_j \cap \Omega_{\rho} ) = J_0(u, \Omega_{\rho}).
\end{equation}
Thus, by \eqref{LiminfIneqEq3} and \eqref{LiminfIneqEq4}, passing the limit as $ \rho $ tends to $ 1 $ we have the liminf inequality
\begin{equation}\label{LiminfIneqEq5}
\liminf_{\varepsilon \rightarrow 0} \tilde{J}_{\varepsilon}(u_{\varepsilon}, \Omega) \geq J_0 (u, \overline{\Omega}),
\end{equation}
utilizing also the continuity of measures on decreasing sets.
$ \\ $

We now prove the $ \Gamma- $limsup inequality. Let $ u \in BV(\Omega ; \lbrace a_1,a_2,...,a_N \rbrace) $ be such that $ u = g_0 $ on $ \Omega_{\rho_0} \setminus \Omega . 
\\ $
a) We first assume that $ u = g_0 $ on $ \Omega \setminus \Omega_{\rho_1} $ with $ \rho_1 < 1 $ and $ |\rho_1 -1| $ small. 

As we observe in the proof of Theorem 2.5 in \cite{Baldo} the $ \Gamma $-limsup inequality for $ J_\varepsilon $ also holds without the mass constraint, see in particular the proof of Lemma 3.1 in \cite{Baldo}. Since the $ \Gamma $-liminf inequality holds, the $ \Gamma $-limsup inequality is equivalent with
\begin{equation}\label{LimsupIneqEq1}
J_0 (u, \Omega) = \lim_{\varepsilon \rightarrow 0} 
J_{\varepsilon}(u_{\varepsilon},\Omega ),
\end{equation}
for some sequence $ u_{\varepsilon} $ converging to $ u $ in $ L^1(\Omega;\mathbb{R}^m) $. So let $ u_\varepsilon $ be a sequence converging to $ u $ in $ L^1(\Omega_{\rho_1};\mathbb{R}^m) $ such that \eqref{LimsupIneqEq1} is satisfied. In particular $ u_\varepsilon $ converges to $ g_0 $ on $ \Omega\setminus \Omega_{\rho_1} $, where $ \Omega_{\rho_1} $ is a small contraction of $ \Omega $.

Now, utilizing the sequence $ u_\varepsilon $ obtained from \eqref{LimsupIneqEq1}, we will modify it by a cut-off function so that the boundary condition is satisfied. By Lemma \ref{LemmainABPcutoff}, there exist a cut-off function $ \phi $ between $ U= \Omega_{\frac{1+\rho_1}{2}} $ and $ U'= \Omega $ such that
\begin{align}\label{LimsupIneqEq2}
J_{\varepsilon}(u_\varepsilon \phi + (1- \phi )g_\varepsilon, \Omega) \leq J_{\varepsilon}(u_\varepsilon, \Omega) + J_{\varepsilon} (g_\varepsilon, V) + \delta_{\varepsilon} (u_\varepsilon, g_\varepsilon, U,U',V),
\end{align}
where $ V = \Omega \setminus \overline{\Omega}_{\rho_1} $ and $ g_\varepsilon $ is extended in $ V $ trivially.

By the assumptions on $ u_\varepsilon $ and \textbf{(H2)} we also have
\begin{align*}
u_\varepsilon \rightarrow g_0 \;,\;\;\;\;\; g_\varepsilon \rightarrow g_0 \;\;\; \textrm{in} \;\: L^1(V).
\end{align*}
Hence, again by Lemma \ref{LemmainABPcutoff} we get
\begin{align*}
\lim_{\varepsilon \rightarrow 0} \delta_{\varepsilon} (u_\varepsilon, g_\varepsilon, U,U',V) =0.
\end{align*}
Note that the condition $ \sup_{\varepsilon>0}( J_\varepsilon (u_\varepsilon,U') + J_\varepsilon(g_\varepsilon,V)) < + \infty $ in Lemma \ref{LemmainABPcutoff} is satisfied. To be more precise, from Lemma \ref{Lemma2} it holds
\begin{align*}
\sup_{\varepsilon>0} J_{\varepsilon} (u_\varepsilon, U') < + \infty \;\;,\; \textrm{where} \;\: U' = \Omega,
\end{align*}
and by \textbf{(H2)(i)},
\begin{align*}
\sup_{\varepsilon>0} J_{\varepsilon} (g_\varepsilon, V) < + \infty \;\;,\; \textrm{where} \;\: V = \Omega \setminus \overline{\Omega}_{\rho_1}.
\end{align*}
So, by \eqref{LiminfIneqEq1}, \eqref{LiminfIneqEq2} and \eqref{LimsupIneqEq2}
\begin{align*}
\Gamma - \limsup_{\varepsilon \rightarrow 0} \tilde{J}_\varepsilon ( \tilde{u}_{\varepsilon}, \Omega) \leq \tilde{J}_0(u, \Omega),
\end{align*}
where $ \tilde{u}_{\varepsilon}= u_\varepsilon \phi + (1- \phi)g_\varepsilon $ and $ \tilde{u}_{\varepsilon}= g_\varepsilon $ in $ \Omega_{\rho_0} \setminus \Omega . \\ $
b) In the general case we consider $ \rho_1 <1 $ and we define $ u_{\rho_1}(x) = u(\frac{1}{\rho_1} x) $ and without loss of generality we may asume that the origin of $ \mathbb{R}^n $ belongs in $ \Omega $. 

By the previous case (a) and \eqref{LimitingEneryFunctional},
\begin{align}\label{LimsupIneqEq3}
\begin{gathered}
\Gamma - \limsup_{\varepsilon \rightarrow 0} \tilde{J}_\varepsilon (u_{\rho_1}, \Omega) \leq \tilde{J}_0(u_{\rho_1}, \Omega) =   
\sum_{1 \leq i < j \leq N} \sigma_{ij} \mathcal{H}^{n-1} (S_{ij}(u_{\rho_1}) \cap \Omega)
\\ \leq \sum_{1 \leq i < j \leq N} \sigma_{ij} \mathcal{H}^{n-1} (S_{ij}(u) \cap \overline{\Omega}) + O(1 - \rho_1^{n-1}) \\ = \tilde{J}_0 (u, \overline{\Omega}) + O(1- \rho_1^{n-1}).
\end{gathered}
\end{align}
Since $ u_{\rho_1} $ converges to $ u $ as $ \rho_1 $ tends to $ 1 $, if we denote
\begin{align*}
J'(u_{\rho_1}, \Omega) := \Gamma - \limsup_{\varepsilon \rightarrow 0} \tilde{J}_\varepsilon (u_{\rho_1}, \Omega),
\end{align*}
then by the lower semicontinuity of the $ \Gamma -$upper limit (see e.g. Proposition 1.28 in \cite{Braides}) and \eqref{LimsupIneqEq3},
\begin{equation}\label{LimsupIneqEq4}
\Gamma - \limsup_{\varepsilon \rightarrow 0} \tilde{J}_\varepsilon (u_{\rho_1}, \Omega) \leq \liminf_{\rho_1 \rightarrow 1} J'(u_{\rho_1}, \Omega) \leq \tilde{J}_0 (u, \overline{\Omega}).
\end{equation}
Hence by \eqref{LiminfIneqEq5} and \eqref{LimsupIneqEq4} we get the required equality \eqref{GammaLimitwithBC}.
\end{proof}
$ \\ $

\section{Minimizing partitions and the structure of the minimizer}
%$ \\ $

In this section we begin with the basic definitions of minimizing partitions. Then we underline the relationship of minimizing partitions in $ \mathbb{R}^2 $ with the minimizers of the functional $ \tilde{J}_0 $ and 
we analyze the structure of the minimizer of $ \tilde{J}_0 $ that we obtain from the $ \Gamma $-limit. Utilizing a Bernstein type theorem for minimizing partitions we can explicitly compute the energy of the minimizer in Proposition \ref{PropositionEnergyEqual3} and by regularity results in \cite{Morgan} we can determine the precise structure of a minimizer subject to the limiting boundary conditions in Theorem \ref{TheoremTriod} and prove uniqueness. In subsection \ref{subsectionMinDim3} we make some comments for the limiting minimizers in dimension three. Finally, in the last subsection we note that we can extend these results to the mass constraint case. $ \\ $

Let $ \Omega \subset \mathbb{R}^n $ open, occupied by $ N $ phases. Associated to each pair of phases $ i $ and $ j $ there is a surface energy density $ \sigma_{ij} $, with $ \sigma_{ij}>0 $ for $ i \neq j $ and $ \sigma_{ij} = \sigma_{ji} $, with $ \sigma_{ii}=0 $. Hence, if $ A_i $ denoted the subset of $ \Omega $ occupied by phase $ i $, then $ \Omega $ is the disjoint union
\begin{align*}
\Omega = A_1 \cup A_2 \cup ... \cup A_N
\end{align*}
and the energy of the partition $ A = \lbrace A_i \rbrace_{i=1}^N $ is
\begin{align}\label{EnergyofPartition}
E(A) = \sum_{1 \leq i < j \leq N} \sigma_{ij} \mathcal{H}^{n-1}(\partial^* A_i \cap \partial^* A_j),
\end{align}
where $ \mathcal{H}^{n-1} $ is the $ (n-1) $-Hausdorff measure in $ \mathbb{R}^n $ and $ A_i $ are sets of finite perimeter. If $ \Omega $ is unbounded, for example $ \Omega =\mathbb{R}^n $ (we say then that $ A $ is complete), the quantity above in general will be infinity. Thus, for each $ W $ open, with $ W \subset \subset \Omega $, we consider the energy
\begin{equation}\label{EnergyofPartitionLocal}
E(A;W)=  \sum_{0< i < j \leq N} \sigma_{ij} \mathcal{H}^{n-1}(\partial^* A_i \cap \partial^* A_j \cap W).
\end{equation}
$ \\ $

\begin{definition}\label{DefinitionMinPar} The partition  $ A $ is a \textit{minimizing} $ N $-partition if given any $ W \subset \subset \Omega $ and any $ N $-partition $ A' $ of $ \Omega $ with
\begin{equation}
\bigcup_{i=1}^N (A_i \triangle A_i') \subset \subset W,
\end{equation}
we have
\begin{align*}
E(A;W) \leq E(A';W).
\end{align*}
\end{definition}
$ \\ $
The symmetric difference $ A_i \triangle A_i' $ is defined as their union minus their intersection, that is, $ A_i \triangle A_i' = (A_i \cup A_i') \setminus (A_i \cap A_i'). 
\\ $

To formulate the Dirichlet problem, we assume that $ \partial \Omega $ is $ C^1 $ and given a partition $ C $ of $ \partial \Omega $ up to a set of $ \mathcal{H}^{n-1} $-measure zero, we may prescribe the boundary data for $ A $:
\begin{align*}
(\partial_{\Omega} A)_i = \partial A_i \cap \partial \Omega = C_i \;,\;\;\;\; i=1,...,N.
\end{align*}
Now the energy is minimized subject to such a prescribed boundary. $ \\ $

\begin{remark}\label{Remark2} Note that the minimization of the functional $ \tilde{J}_0(u, \Omega) $ is equivalent to minimizing the energy $ E(A;\Omega) $ under the appropriate Dirichlet conditions. 
\end{remark}

$ \\ $

\tikzset{every picture/.style={line width=0.75pt}} %set default line width to 0.75pt        

\begin{tikzpicture}[x=0.75pt,y=0.75pt,yscale=-1,xscale=1]

%Straight Lines [id:da3457011418365854] 
\draw    (164,160) -- (256.8,249) ;
%Straight Lines [id:da08920364473044984] 
\draw    (133.8,52) -- (164,160) ;
%Straight Lines [id:da9104522195533162] 
\draw    (58.8,234) -- (164,160) ;
%Shape: Triangle [id:dp6511635950211963] 
%\draw   (483.17,87.95) -- (601.8,200) -- (421.65,197.33) -- cycle ;
\draw   (393.17,87.95) -- (511.8,200) -- (331.65,197.33) -- cycle ;

% Text Node
\draw (219,105.4) node [anchor=north west][inner sep=0.75pt]    {$1$};
% Text Node
\draw (66,108.4) node [anchor=north west][inner sep=0.75pt]    {$2$};
% Text Node
\draw (154,233.4) node [anchor=north west][inner sep=0.75pt]    {$3$};
% Text Node
\draw (178,139.4) node [anchor=north west][inner sep=0.75pt]    {$\theta _{1}$};
% Text Node
\draw (131,131.4) node [anchor=north west][inner sep=0.75pt]    {$\theta _{2}$};
% Text Node
\draw (149,184.4) node [anchor=north west][inner sep=0.75pt]    {$\theta _{3}$};
% Text Node
\draw (276,268) node [anchor=north west][inner sep=0.75pt]   [align=left] {Figure 1.};
% Text Node
\draw (461,168.4) node [anchor=north west][inner sep=0.75pt]    {$\hat{\theta }_{1}$};
% Text Node
\draw (388,98.4) node [anchor=north west][inner sep=0.75pt]    {$\hat{\theta }_{2}$};
% Text Node
\draw (352,167.4) node [anchor=north west][inner sep=0.75pt]    {$\hat{\theta }_{3}$};
% Text Node
\draw (460,112.4) node [anchor=north west][inner sep=0.75pt]    {$\sigma _{1}{}_{2}$};
% Text Node
\draw (410,210.4) node [anchor=north west][inner sep=0.75pt]    {$\sigma _{1}{}_{3}$};
% Text Node
\draw (329,115.4) node [anchor=north west][inner sep=0.75pt]    {$\sigma _{2}{}_{3}$};

\end{tikzpicture}
$ \\ $

In Figure 1 we show a triod with angles $ \theta_1, \theta_2, \theta_3 $, and the corresponding triangle with their supplementary angles $ \hat{\theta}_i = \pi - \theta_i $ . For these angles Young's law holds, that is,
\begin{equation}\label{YoungLaw}
\frac{\textrm{sin}\hat{\theta}_1}{\sigma_{23}} = \frac{\textrm{sin}\hat{\theta}_2}{\sigma_{13}} = \frac{\textrm{sin}\hat{\theta}_3}{\sigma_{12}}.
\end{equation}
$ \\ $

\begin{definition}\label{DefinitionTriod}
Let $ \mathcal{A}_{x_0} = \lbrace A_1,A_2,A_3 \rbrace $ be a $ 3 -$partition of $ \mathbb{R}^2 $ such that $ A_i $ is a single infinite sector emanating from the point $ x_0 \in \mathbb{R}^2 $ with three opening angles $ \theta_i $ that satisfy \eqref{YoungLaw}. We call as a \textit{triod} $ C_{tr}(x_0) $ the boundary of the partition $ \mathcal{A}_{x_0} $, that is, $ C_{tr}(x_0) = \lbrace \partial A_i \cap \partial A_j \rbrace_{1 \leq i < j \leq 3} $.
\end{definition}

So, in other words, the triod is consisted of three infinite lines meeting at a point $ x_0 $ and their angles between the lines satisfy the Young's law \eqref{YoungLaw} (see Figure 1). As we see in Theorem \ref{BernsteinThmforPartitionsinthePlane}, the triod is the unique locally $ 3 -$minimizing partition of $ \mathbb{R}^2 $. The point $ x_0 $, i.e. the center of the triod, is often called a \textit{triple junction point}.

$ \\ $

\subsection{The structure of the minimizer in the disk}

Throughout this section we will assume that $ \sigma_{ij} = \sigma >0 $ for $ i \neq j $, therefore we have by Young's law $ \theta_i = \dfrac{2 \pi}{3} \;\:,\; i=1,2,3 $. 
%That is one main reason for imposing such particular Dirichlet conditions in \textbf{(H2)}. 
As a result of Theorem \ref{BernsteinThmforPartitionsinthePlane}, we expect that, by imposing the appropriate boundary conditions, the minimizer $ u_0 $ of $ \tilde{J}_0(u, \overline{B}_1) \;,\; B_1 \subset \mathbb{R}^2 $ which we obtain from the $ \Gamma $-limit will be a triod with angles $ \dfrac{2 \pi}{3} $ restricted in $ B_1 $ and centered at a point $ x \in B_1 . \\ $

We now recall \textit{Steiner's problem} that gives us some geometric intuition about this fact.

Let us take three points $ A , \: B  $ and $ C $, arranged in any way in the plane. The problem is to find a fourth point $ P $ such that the sum of distances from $ P $ to the other three points is a minimum; that is we require $ AP + BP + CP $ to be a minimum length.

If the triangle $ ABC $ possesses internal angles which are all less than $ 120^o $, then $ P $ is the point such that each side of the triangle, i.e. $ AB ,\: BC $ and $ CA $, subtends an angle of $ 120^o $ at $ P $. However, if one angle, say $ A \hat{C} B $, is greater than $ 120^o $, then $ P $ must coincide with $ C . $

The \textit{Steiner's problem} is a special case of the Geometric median problem and has a unique solution whenever the points are not collinear. For more details and proofs see \cite{GP}. $ \\ $

%$ g^{\varepsilon}(1,\theta)= a_1 \;,\;\: \theta \in (C \varepsilon, \frac{2\pi}{3} - C\varepsilon) \;,\; g^{\varepsilon}(1,\theta) = a_2 \;,\: \theta \in ( \frac{2\pi}{3} + C\varepsilon,  \frac{4 \pi}{3} - C\varepsilon) $ and $ g^{\varepsilon}(1,\theta) = a_3 \;,\; \theta \in (\frac{4 \pi}{3} + C\varepsilon, 2 \pi - C\varepsilon) $ (in polar coordinates) and connecting $ a_1 $ with $ a_2 $ in $ (\frac{2\pi}{3} - C\varepsilon, \frac{2\pi}{3} + C\varepsilon) $ and similarly in the remaining intervals.
$ \\ $

The problem of minimizing partitions subject to boundary conditions, in contrast to the mass constraint case, might not always admit a minimum, we provide an example in Figure 2 below.

%$ \\ $

\begin{center}

\tikzset{every picture/.style={line width=0.75pt}} %set default line width to 0.75pt        

\begin{tikzpicture}[x=0.75pt,y=0.75pt,yscale=-1,xscale=1]

%Shape: Arc [id:dp22758392956586104] 
\draw  [draw opacity=0] (354.38,69.55) .. controls (366.72,81.81) and (374.31,99.25) .. (374.03,118.5) .. controls (373.5,155) and (344.84,184.18) .. (310.03,183.69) .. controls (275.22,183.19) and (247.43,153.19) .. (247.97,116.69) .. controls (248.25,97.57) and (256.25,80.45) .. (268.81,68.56) -- (311,117.59) -- cycle ; \draw   (354.38,69.55) .. controls (366.72,81.81) and (374.31,99.25) .. (374.03,118.5) .. controls (373.5,155) and (344.84,184.18) .. (310.03,183.69) .. controls (275.22,183.19) and (247.43,153.19) .. (247.97,116.69) .. controls (248.25,97.57) and (256.25,80.45) .. (268.81,68.56) ;  
%Straight Lines [id:da3240953908600994] 
\draw    (268.81,68.56) -- (354.38,69.55) ;

% Text Node
\draw (304,98.4) node [anchor=north west][inner sep=0.75pt]    {$\Omega $};
% Text Node
\draw (244,41.4) node [anchor=north west][inner sep=0.75pt]    {$A$};
% Text Node
\draw (360,43.4) node [anchor=north west][inner sep=0.75pt]    {$B$};
% Text Node
\draw (305,32.4) node [anchor=north west][inner sep=0.75pt]    {$a_{2}$};
% Text Node
\draw (288,194.4) node [anchor=north west][inner sep=0.75pt]    {$g_{0} =a_{1}$};
% Text Node
\draw (284,253) node [anchor=north west][inner sep=0.75pt]   [align=left] {Figure 2.};

\end{tikzpicture}
\end{center}

%$ \\ $

However a minimizer will exist for the minimization problem $ \min_{u\in BV(\Omega; \lbrace W=0 \rbrace)} \tilde{J}_0(u, \overline{\Omega}) $, for instance the one we obtain from the $ \Gamma $-limit, which will form a ``boundary layer'' in the boundary of the domain instead of internal layer (i.e. the interface separating the phases). Particularly, in Figure 2 above, $ u_0 = a_1 $, a.e. will be a minimizer of $ \tilde{J}_0 $ and 
\begin{align*}
\tilde{J}_0(u_0,\overline{\Omega}) = \frac{1}{2} \sum_{i=1}^3 \int_{\partial \Omega} | T(\phi_i \circ u_0) - T ( \phi_i \circ g_0) | d\mathcal{H}^1 = \sigma \mathcal{H}^1(\partial \Omega_{AB}) ,
\end{align*}
where $ \partial \Omega_{AB} $ is the part of the boundary of $ \Omega $ in which $ g_0 = a_2 $. When there are no line segments in the boundary of the domain or when $ g_0 $ does not admit jumps nearby such line segments, then we expect that there are no boundary layers and the boundary term in the energy of $ \tilde{J}_0 $ vanishes (see Remark \ref{Remark1}), otherwise we could find a minimizer with strictly less energy. In the cases where the boundary term vanishes we can write $ \tilde{J}_0(u_0, \overline{\Omega})  = \tilde{J}_0(u_0,\Omega) $. This can be proved rigorously in the case where $ \Omega = B_1 $ and assuming \textbf{(H2)(iii)}, utilizing also Proposition \ref{PropMorgan} as we will see in the proof of Theorem \ref{TheoremTriod}.
$ \\ $

\begin{remark}\label{Remark3} For the mass constraint case, by classical results of Almgren's improved and simplified by Leonardi in \cite{Leonardi} for minimizing partitions with surface tension coefficients $ \sigma_{ij} $ satisfying the strict triangle inequality (see \eqref{triangleinequality}), $ \Omega_j $ can be taken open with $ \partial \Omega_j $ real analytic except possibly for a singular part with Hausdorff dimension at most $ n-2 $. Therefore $ \partial^* \Omega_i \cap \partial^* \Omega_j = \partial \Omega_i \cap \partial \Omega_j  \;,\: \mathcal{H}^{n-1} $-a.e., where $ u_0 = \sum_{i=1}^N a_i \chi_{\Omega_i} $ is the minimizer of $ J_0 $ with a mass constraint. These regularity results have been stated by White in \cite{White} but without providing a proof. Also, Morgan in \cite{Morgan2} has proved regularity of minimizing partitions in the plane subject to mass constraint. However, we deal with the problem with boundary conditions, so we cannot apply these regularity results.
\end{remark}

$ \\ $

\textbf{Notation:} We set as $ x_0 \in B_1 $ the point such that the line segments starting from $ p_i = \partial I_k \cap \partial I_l \;,\: k \neq l \;,\: i \in \lbrace 1,2,3 \rbrace \setminus \lbrace k,l \rbrace $ and ending at $ x_0 $ meet all at angle $ \frac{2 \pi}{3} $ (see \textbf{(H2)(iii)} and Proposition \ref{PropMorgan}). Also we denote by $ C_0 $ the sum of the lengths of these line segments. The following Proposition measures the energy of the limiting minimizer. $ \\ $

\begin{proposition}\label{PropositionEnergyEqual3} Let $ (u_\varepsilon) $ be a minimizing sequence of $ \tilde{J}_\varepsilon (u,B_1) $. Then $ u_\varepsilon \rightarrow u_0 $ in $ L^1 $ along subsequence with $ u_0 \in BV(B_1; \lbrace a_1,a_2,a_3 \rbrace) $ and $ u_0 $ is a minimizer of $ \tilde{J}_0(u, \overline{B}_1) $  subject to the limiting Dirichlet values \textbf{(H2)(iii)}, where we extend $ u $ by setting $ u=g_0 $ on $ \mathbb{R}^2 \setminus B_1 $.

In addition, we have
\begin{align}\label{InterfacesEqual3}
\sum_{1 \leq i <j \leq 3} \mathcal{H}^1 (\partial^* \Omega_i \cap \partial^* \Omega_j \cap \overline{B}_1 ) = C_0 \;,
\end{align}
where $ u_0 = a_1 \chi_{\Omega_1} +a_2 \chi_{\Omega_2} + a_3 \chi_{\Omega_3}. $
\end{proposition}
$ \\ $

\begin{proof} $ \;\: $
From Lemma \ref{Lemma2}, \ref{Lemma3} it holds that if $ u_\varepsilon $ is a minimizing sequence for $ \tilde{J}_\varepsilon(u,B_1) $, then $ \tilde{J}_\varepsilon(u_\varepsilon,B_1) \leq C $ and thus $ u_\varepsilon \rightarrow u_0 $ in $ L^1 $ along subsequence. The fact that $ u_0 $ is a minimizer of $ \tilde{J}_0 $ is a standard fact from the theory of $ \Gamma- $convergence. 
It can be seen as follows.

Let $ w \in BV(\overline{B_1}, \lbrace a_1,a_2,a_3 \rbrace )$ such that $ w = g_0 $ on $ \mathbb{R}^2 \setminus B_1 $, then from the limsup inequality in Theorem \ref{Theorem1}, we have that there exists $ w_\varepsilon \in H^1_{loc}(\mathbb{R}^2 ;\mathbb{R}^m) \;\:,\; w_\varepsilon = g_\varepsilon $ on $ \mathbb{R}^2 \setminus B_1 $ such that $ w_\varepsilon \rightarrow w $ in $ L^1 $ and $ \limsup_{\varepsilon \rightarrow 0} \tilde{J}_\varepsilon (w_\varepsilon ,B_1) \leq \tilde{J}_0 (w, \overline{B}_1) $. Now since $ u_\varepsilon $ is a minimizing sequence for $ \tilde{J}_\varepsilon(u,B_1) $ and from the liminf inequality in Theorem \ref{Theorem1}, we have
\begin{equation}\label{LimitingMinimality}
\begin{gathered}
\tilde{J}_0 (u_0 , \overline{B}_1) \leq \liminf_{\varepsilon \rightarrow 0} \tilde{J}_\varepsilon (u_\varepsilon, B_1) \leq \liminf_{\varepsilon \rightarrow 0} \tilde{J}_\varepsilon ( w_\varepsilon,B_1) \\ \leq \limsup_{\varepsilon \rightarrow 0} \tilde{J}_\varepsilon ( w_\varepsilon,B_1) \leq \tilde{J}_0 (w, \overline{B}_1)
\end{gathered}
\end{equation}
$ \\ $

For proving \eqref{InterfacesEqual3}, we utilize Theorem \ref{BernsteinThmforPartitionsinthePlane} (i.e. Theorem 2 in \cite{A}). Since the triod is a minimizing 3-partition in $ \mathbb{R}^2 $ we have that for any $ W \subset \subset \mathbb{R}^2 $ and any partition it holds that $ E(A,W) \leq E(V,W) $, where suppose that $ A = \lbrace A_1,A_2,A_3 \rbrace $ is the partition of the triod and $ V = \lbrace V_1,V_2,V_3 \rbrace $ is a 3-partition in $ \mathbb{R}^2 . $

We have $ u_0 = a_1 \chi_{\Omega_1} +a_2 \chi_{\Omega_2} + a_3 \chi_{\Omega_3} $ such that $ u_0 = g_0 $ on $ \partial B_1 $ and extend $ u_0 $ in $ \mathbb{R}^2 $, being the triod with $ \theta_i = \dfrac{2 \pi}{3} $ in $ \mathbb{R}^2 \setminus B_1 $ centered at $ x_0 $. This defines a 3-partition in $ \mathbb{R}^2 $, noted as $ \tilde{\Omega} = \lbrace \tilde{\Omega}_i \rbrace_{i=1}^3 $. Since the triod is a minimizing 3-partition in the plane, we take any $ W \subset \subset \mathbb{R}^2 $ such that $ B_2 \subset \subset W $ and $ \bigcup_{i=1}^3 (A_i \triangle \tilde{\Omega}_i) \subset \subset W $, so we have
\begin{equation}\label{ComparisonwithTriod}
E(A, W) = E(A,\overline{B}_1) + E(A , W \setminus \overline{B}_1) \leq E( \tilde{\Omega},W) = E( \tilde{\Omega} , \overline{B}_1) + E( \tilde{\Omega} ,W \setminus \overline{B}_1)
\end{equation}
where $ A $ is the partition of the triod.

Now since
\begin{align*}
 E(A , W \setminus \overline{B}_1) = E( \tilde{\Omega} ,W \setminus \overline{B}_1)
\end{align*}
from the way we extended $ u_0 $ in $ \mathbb{R}^2 $ and 
\begin{align*}
E(A , \overline{B}_1) = \sigma \sum_{1 \leq i < j \leq 3} \mathcal{H}^1 ( \partial A_i \cap \partial A_j \cap \overline{B}_1) = C_0 \sigma 
\end{align*}
since $ \partial A_i \cap \partial A_j \cap \overline{B}_1 $ are line segments inside $ B_1 $ with sum of their lengths equals $ C_0 $, we conclude
\begin{equation}\label{Lowerinequality}
\begin{gathered}
C_0 \sigma \leq E( \tilde{\Omega} , \overline{B}_1) = \tilde{J}_0 (u_0, \overline{B}_1) \\
\Leftrightarrow C_0 \leq \sum_{1 \leq i < j \leq 3} \mathcal{H}^1(\partial^* \Omega_i \cap \partial^* \Omega_j \cap \overline{B}_1)
\end{gathered}
\end{equation}

For the upper bound inequality $ \sum_{1 \leq i < j \leq 3} \mathcal{H}^1(\partial^* \Omega_i \cap \partial^* \Omega_j \cap \overline{B}_1) \leq C_0 $, we consider as a comparison function $ \tilde{u}= a_1 \chi_{A_1} +a_2 \chi_{A_2} + a_3 \chi_{A_3} $, where $ C_{tr}(x_0)= \lbrace A_1,A_2,A_3 \rbrace $ is the partition of the triod centered at $ x_0 \in B_1 $ and angles $ \theta_i = \frac{2 \pi}{3} $ (see Definition \ref{DefinitionTriod}).

Then $ \tilde{u} $ satisfies the boundary condition $ \tilde{u} = g_0 $ on $ \mathbb{R}^2 \setminus B_1 $ and therefore by the minimality of $ u_0 $ we have
\begin{equation}\label{UpperBoundIneqProof}
\begin{gathered}
\tilde{J}_0 (u_0, \overline{B}_1) \leq \tilde{J}_0 ( \tilde{u}, \overline{B}_1) = C_0 \sigma
\\ \Rightarrow \sum_{1 \leq i < j \leq 3} \mathcal{H}^1(\partial^* \Omega_i \cap \partial^* \Omega_j \cap \overline{B}_1) \leq C_0.
\end{gathered}
\end{equation}
\end{proof}
%$ \\ $
%Proposition \ref{PropositionEnergyEqual3} above holds also in the case where $ \sigma_{ij} $ (i.e. the energy of the connection that connects $ a_i $ to $ a_j $) are not all equal, but satisfy the strict triangle inequality \eqref{triangleinequality}. In this case however, instead of angle of $ \dfrac{2 \pi}{3} $, we have angles $ \theta_i $ obtained from the Young's law \eqref{YoungLaw}.

\begin{corollary}\label{CorDensityofPartition}
Assume for simlicity that $ x_0 $ in Proposition \ref{PropositionEnergyEqual3} above is the origin of $ \mathbb{R}^2 $. Then for every $ R>0 $ the energy of the limiting minimizer will satisfy
\begin{equation}\label{CorEnergyofLimitingMinimizer}
\tilde{J}_0 (u_0, \overline{B}_R) = 3 \sigma R.
\end{equation}
In addition, there exists an entire minimizer in the plane and the partition that defines is a minimal cone.
\end{corollary}
\begin{proof} $ \;\: $
Since $ x_0 $ is the origin of $ \mathbb{R}^2 $, it holds that $ C_0 $ in \eqref{InterfacesEqual3} equals $ 3 $. Arguing as in Proposition \ref{PropositionEnergyEqual3} above we can similarly obtain a minimizer of $ \tilde{J}_0 (u_0, \overline{B}_R) $ that satisfies \eqref{CorEnergyofLimitingMinimizer}. By a diagonal argument the minimizer can be extended in the entire plane and will also satisfy
\begin{align*}
\dfrac{\mathcal{H}^1(\partial \Omega_i \cap \partial \Omega_j \cap B_R)}{\omega_1 R} = C \;\;\:,\; \forall \; R >0.
\end{align*}
Thus, the partition that it defines is a minimal cone (see \cite{WhiteNotes} or \cite{A}).
\end{proof}

$ \\ $

Finally, we will prove that the minimizer of $ \tilde{J}_0 $ in $ \overline{B}_1 $ is unique, that is, the only minimizer is the triod restricted to $ B_1 $ centered at a point in $ B_1 . $ In Figure 3 below we provide the structure of the minimizer $ u_0 $ obtained in Theorem \ref{TheoremTriod}. $ \\ $

\begin{center}
\tikzset{every picture/.style={line width=0.75pt}} %set default line width to 0.75pt        

\begin{tikzpicture}[x=0.75pt,y=0.75pt,yscale=-1,xscale=1]

%Shape: Circle [id:dp36891925175586526] 
\draw   (296,137) .. controls (296,93.92) and (330.92,59) .. (374,59) .. controls (417.08,59) and (452,93.92) .. (452,137) .. controls (452,180.08) and (417.08,215) .. (374,215) .. controls (330.92,215) and (296,180.08) .. (296,137) -- cycle ;
%Straight Lines [id:da9204174536839491] 
\draw    (338.3,69) -- (374.3,97) ;
%Straight Lines [id:da523043440784388] 
\draw    (374.3,97) -- (374,215) ;
%Straight Lines [id:da04384785595814811] 
\draw    (374.3,97) -- (410.3,69) ;
%Curve Lines [id:da563606849892101] 
\draw    (375.3,108) .. controls (361.3,113) and (356.3,95) .. (363.3,88) ;
%Curve Lines [id:da36972495594193355] 
\draw    (384.3,89) .. controls (396.3,98) and (385.3,113) .. (375.3,108) ;

% Text Node
\draw (322,41.4) node [anchor=north west][inner sep=0.75pt]    {$A$};
% Text Node
\draw (408,45.4) node [anchor=north west][inner sep=0.75pt]    {$B$};
% Text Node
\draw (368,217.4) node [anchor=north west][inner sep=0.75pt]    {$C$};
% Text Node
\draw (394.29,86.41) node [anchor=north west][inner sep=0.75pt]  [font=\small,rotate=-0.16]  {$\frac{2\pi }{3}$};
% Text Node
\draw (333.29,85.41) node [anchor=north west][inner sep=0.75pt]  [font=\small,rotate=-0.16]  {$\frac{2\pi }{3}$};
% Text Node
\draw (349,34.4) node [anchor=north west][inner sep=0.75pt]    {$g_{0} =a_{1}$};
% Text Node
\draw (454,154.4) node [anchor=north west][inner sep=0.75pt]    {$g_{0} =a_{2}$};
% Text Node
\draw (240,143.4) node [anchor=north west][inner sep=0.75pt]    {$g_{0} =a_{3}$};
% Text Node
\draw (343,259) node [anchor=north west][inner sep=0.75pt]   [align=left] {Figure 3.};
\end{tikzpicture}
\end{center}

\begin{proof}[Proof of Theorem~\ref{TheoremTriod}] $ \;\: $
Firstly, we show that the minimizing partition of $ B_1 $ with respect to the boundary conditions defined from $ g_0 $, is a $ (M,0,\delta) $-minimal for $ \delta>0 $ 
%and therefore $ (M,cr^{\alpha}, \delta) $-minimal 
(see Definition \ref{M,0,deltaminimalSets}). If not, let $ S $ be the partition defined from $ u_0 $, we can find a Lipschitz function $ \phi : \mathbb{R}^2 \rightarrow \mathbb{R}^2 $ such that
\begin{align*}
\mathcal{H}^1 (S \cap W) > \mathcal{H}^1( \phi(S \cap W)),
\end{align*}
with 
\begin{align*}
 W = \mathbb{R}^2 \cap \lbrace x : \phi(x) \neq x \rbrace \;\:,\; \textrm{diam}(W \cup \phi(W)) < \delta \\ \textrm{and dist}(W \cup \phi(W), \mathbb{R}^2 \setminus B_1) >0. \;\;\;\;\;\;\;\;\;\;\;\;\;\;
\end{align*}

So if we consider
the partition
\begin{align*}
\tilde{S} := \begin{cases} S \;\;,\; S \cap W = \emptyset \\ \phi(S \cap W) \;\;,\; S \cap W \neq \emptyset
\end{cases} ,
\end{align*}
then the boundary of the partition defined by $ \tilde{S} $ will satisfy the boundary conditions (since dist$ (W \cup \phi(W), \mathbb{R}^2 \setminus B_1) >0 $) and also $ \mathcal{H}^1 (\tilde{S}) < \mathcal{H}^1(S) $ which contradicts the minimality of $ S . $

Thus, by \textbf{(H2)(iii)} we apply Proposition \ref{PropMorgan} and we have that the unique smallest $ (M,0, \delta) $-minimal set consists of three line segments from the three vertices defined from $ g_0 $ (i.e. the jump points in $ \partial B_1 $) meeting at $ \frac{2 \pi}{3} $. The meeting point is unique and belongs in the interior of $ B_1 $. Thus, $ \partial \Omega_i \cap \partial \Omega_j = \partial^* \Omega_i \cap \partial^* \Omega_j $ are line segments meeting at $ \frac{2 \pi}{3} $ in an interior point of $ B_1 $.
\end{proof}
$ \\ $

\begin{corollary}\label{CorollaryTriod} Let $ u_0 = a_1 \chi_{\Omega_1} +a_2 \chi_{\Omega_2} + a_3 \chi_{\Omega_3} $ be a minimizer of $ \tilde{J}_0(u,\overline{B}_1) $ subject to the limiting Dirichlet values $ g_0( \theta) = a_1 \chi_{(0, \frac{2 \pi}{3})} + a_2 \chi_{(\frac{2 \pi}{3} , \frac{4 \pi}{3})} +a_3 \chi_{(\frac{4 \pi}{3} , 2 \pi)} \;,\; \theta \in (0, 2\pi ) . $
Then $ \partial \Omega_i \cap \partial \Omega_j $ are radi of $ B_1 \;,\: | \Omega_i |= \frac{1}{3}| B_1 | $ and the minimizer is unique.
\end{corollary}
%$ \\ $
\begin{center}
\tikzset{every picture/.style={line width=0.75pt}} %set default line width to 0.75pt        

\begin{tikzpicture}[x=0.75pt,y=0.75pt,yscale=-1,xscale=1]

%Shape: Circle [id:dp03083930988889949] 
\draw   (228.2,131.4) .. controls (228.2,89.76) and (261.96,56) .. (303.6,56) .. controls (345.24,56) and (379,89.76) .. (379,131.4) .. controls (379,173.04) and (345.24,206.8) .. (303.6,206.8) .. controls (261.96,206.8) and (228.2,173.04) .. (228.2,131.4) -- cycle ;
%Straight Lines [id:da3258632182060761] 
\draw    (303.6,56) -- (303.6,131.4) ;
%Straight Lines [id:da06638187319291755] 
\draw    (303.6,131.4) -- (365.8,174.8) ;
%Straight Lines [id:da791367115915087] 
\draw    (303.6,131.4) -- (243.8,176.8) ;
%Curve Lines [id:da6886438171937881] 
\draw    (304,121) .. controls (316.8,116.8) and (321.8,131.8) .. (313.8,137.8) ;

% Text Node
\draw (293,22.4) node [anchor=north west][inner sep=0.75pt]    {$A$};
% Text Node
\draw (387,185.4) node [anchor=north west][inner sep=0.75pt]    {$B$};
% Text Node
\draw (210,197.4) node [anchor=north west][inner sep=0.75pt]    {$C$};
% Text Node
\draw (327,85.4) node [anchor=north west][inner sep=0.75pt]    {$\Omega _{1}$};
% Text Node
\draw (295,163.4) node [anchor=north west][inner sep=0.75pt]    {$\Omega _{2}$};
% Text Node
\draw (254,85.4) node [anchor=north west][inner sep=0.75pt]    {$\Omega _{3}$};
% Text Node
\draw (386,85.4) node [anchor=north west][inner sep=0.75pt]    {$g_{0} =a_{1}$};
% Text Node
\draw (273,215.4) node [anchor=north west][inner sep=0.75pt]    {$g_{0} =a_{2}$};
% Text Node
\draw (165,102.4) node [anchor=north west][inner sep=0.75pt]    {$g_{0} =a_{3}$};
% Text Node
\draw (272,252) node [anchor=north west][inner sep=0.75pt]   [align=left] {Figure 4.};
% Text Node
\draw (320,111.4) node [anchor=north west][inner sep=0.75pt]  [font=\small]  {$\frac{2\pi }{3}$};
\end{tikzpicture}
\end{center}

In Figure 4 above we illustrate the structure of the minimizer $ u_0 $ obtained in Corollary \ref{CorollaryTriod}. $ \\ $

\subsection{Minimizers in dimension three}\label{subsectionMinDim3}

In this subsection we will briefly make some comments for the structure of minimizers in $ \mathbb{R}^3 $. If we impose the appropriate boundary conditions in $ B_R \subset \mathbb{R}^3 $ and $ \lbrace W=0 \rbrace = \lbrace a_1,a_2,a_3 \rbrace \;,\: g_\varepsilon \rightarrow g_0 \;\: \textrm{in} \;\: L^1(B_R; \mathbb{R}^3) $ such that the partition in $ \partial B_R $ defined by $ g_0 $ is equal to the partition of $ (C_{tr} \times \mathbb{R}) \cap \partial B_R $, where $ C_{tr} $ is the triod as in Figure 1 (with equal angles), then by Theorem 3 in \cite{A}, arguing as in Proposition \ref{PropositionEnergyEqual3} (see also Corollary \ref{CorDensityofPartition}), we can obtain 
\begin{align*}
\tilde{J}_0 (u, B_R) = \frac{3}{2} \sigma \pi R^2 \;,
\end{align*}
which gives
\begin{align*}
\frac{ \mathcal{H}^2 ( \partial \Omega_i \cap \partial \Omega_j \cap B_R)}{\omega_2 R^2} = \frac{3}{2} \;,
\end{align*}
where $ \omega_2 $ is the volume of the 2-dimensional unit ball (see \cite{WhiteNotes}). That is, the partition that the minimizer defines can be extended to a minimal cone in $ \mathbb{R}^3 $. Now since the only minimizing minimal cones are the triod and the tetrahedral cone (see \cite{Taylor}), then the minimizer of $ \tilde{J}_0 $ is such that $ u_0 = \sum_{i=1}^3 a_i \chi_{\Omega_i} $, where $ \Omega = \lbrace \Omega_i \rbrace_{i=1}^3 $ is the partition of $ (C_{tr} \times \mathbb{R} ) \cap B_R $.

Similarly, if $ \lbrace W=0 \rbrace = \lbrace a_1,a_2,a_3,a_4 \rbrace $ and we impose the Dirichlet conditions such that $ g_0 $ defines the partition of the tetrahedral cone intersection with $ \partial B_R $, then again $ u_0 = \sum_{i=1}^4 a_i \chi_{\Omega_i} $, where $ \Omega = \lbrace \Omega_i \rbrace_{i=1}^4 $ is the partition of the tetrahedral cone restricted in $ B_R $.

$ \\ $

\subsection{Minimizers in the disc for the mass constraint case}

Throughout this subsection we will assume that $ a_i \;,\: i=1,2,3, $ are affinely independent, that is, they are not contained in a single line. This can also be expressed as
\begin{equation}\label{AffineIndependencyofa_i}
\textrm{whenever} \;\: \sum_{i=1}^3 a_i \lambda_i =0 \;\: \textrm{with} \;\: \sum_{i=1}^3 \lambda_i =0 \;,\: \textrm{then} \;\: \lambda_i =0 \;,\: i=1,2,3.
\end{equation}
In addition, we consider that $ m =(m_1,m_2) \in \mathbb{R}^2 $ such that $ m_1,m_2 >0 $ (as in \cite{Baldo}).

Let $ u_0 $ be a minimizer of $ J_0(u,B_1) \;,\: B_1 \subset \mathbb{R}^2 $ defined in \eqref{LimitingEneryFunctional} subject to the mass constraint
\begin{equation}\label{MassConstraint}
 \int_{B_1} u(x) dx = m \;,
\end{equation}
(i.e. the minimizer $ u_0 $ of Theorem p.70 in \cite{Baldo}) and $ \lbrace W=0 \rbrace = \lbrace a_1,a_2,a_3 \rbrace $. Then $ u_0 = \sum_{i=1}^3 a_i \chi_{\Omega_i} $, where $ \Omega_1,\Omega_2,\Omega_3 $ is a partition of $ B_1 $ which minimizes the quantity
\begin{align}\label{J_0MinimizationQuantity}
\sum_{1 \leq i<j \leq 3} \sigma \mathcal{H}^1 (\partial^* \Omega_i \cap \partial^* \Omega_j) \;,
\end{align}
among all other partitions of $ B_1 $ such that $ \sum_{i=1}^3 | \Omega_i | a_i = m . $

\begin{theorem}\label{MinimizerwithMassConstrantUniq}
Let $ u_0 $ be a minimizer of $ J_0(u,B_1) $ as above and assume that 
\begin{align}\label{ConditionOntheconstraintm}
m = \sum_{i=1}^3 c_i a_i \;\;,\; \textrm{where} \;\: c_i >0 \;,\; \textrm{with} \;\: \sum_{i=1}^3 c_i = | B_1 |.
\end{align}
Then
\begin{equation}\label{MinimizerwithMassConstrantUniqEq}
\begin{gathered}
| \Omega_i | = c_i \;,\: i=1,2,3 \;,\;
\partial^* \Omega_i \cap \partial^* \Omega_j = \partial \Omega_i \cap \partial \Omega_j \;\: \textrm{are piecewise smooth} \\ \textrm{and the minimizer is unique up to a rigid motion of the disc}.
\end{gathered}
\end{equation}

In particular, the boundary of the partition is consisted of three circular arcs or line segments meeting at an interior vertex at $ 120 $ degrees angles, reaching orthogonally $ \partial B_1 $ and so that the sum of geodesic curvature is zero.
\end{theorem}
\begin{proof} $ \;\: $
We have that $ u_0 = \sum_{i=1}^3 a_i \chi_{\Omega_i} $, where $ \Omega_i $ are such that $ \sum_{i=1}^3 | \Omega_i | = | B_1 | $ and $ u_0 $ minimizes the quantity \eqref{J_0MinimizationQuantity}.

By the assumption \eqref{ConditionOntheconstraintm}, since $ u_0 $ satisfies \eqref{MassConstraint}, we have
\begin{equation}\label{ProofMinimizerwithMassConstrantUniqEq1}
\begin{gathered}
\sum_{i=1}^3 a_i | \Omega_i | = \sum_{i=1}^3 c_i a_i \;\: \textrm{and} \;\: \sum_{i=1}^3 (  | \Omega_i | -c_i) = 0 \\ \Rightarrow
| \Omega_i | = c_i \;\:,\; i=1,2,3 \;,\; \textrm{and} \;\: c_i \in (0, | B_1 |),
\end{gathered}
\end{equation}
since $ a_i $ are affinely independent.

Now by Theorem 4.1 in \cite{CR} we conclude that the minimizer is a standard graph i.e. it is consisted of three circular arcs or line segments meeting at an interior vertex at $ 120 $ degrees angles, reaching orthogonally $ \partial B_1 $ and so that the sum of geodesic curvature is zero. So, $ \partial^* \Omega_i \cap \partial^* \Omega_j = \partial \Omega_i \cap \partial \Omega_j $ are piecewise smooth.

Finally, the minimizer is unique up to rigid motions of the disc by Theorem 3.6 in \cite{CR}.
\end{proof}

Note that in the case where $ m = \frac{1}{3} | B_1 | \sum_{i=1}^3 c_i a_i $, it holds that $ | \Omega_i | = \frac{1}{3} | B_1 | \;,\; i=1,2,3 \;,\; \textrm{and} \;\: \partial \Omega_i \cap \partial \Omega_j $ are line segments meeting at the origin and the minimizer is unique up to rotations.
$ \\ $

\end{document}